\documentclass{amsart}      
\usepackage{amssymb, amsmath, graphicx}%, showkeys}
\newtoks\prt 
 
\newtheorem{thm}{Theorem}[section] 
\newtheorem{lemma}[thm]{Lemma} 
\newtheorem{prop}[thm]{Proposition} 
 
\newtheorem{example}[thm]{Example}

\theoremstyle{definition}

\def\eqn#1$$#2$${\begin{equation}\label#1#2\end{equation}}

\def\diam{\operatorname{diam}} 
 
\def\ep{\varepsilon} 
 
\def\en{\mathbb N}

\def\r{|} 
 
\def\ov{\overline}

\def\wh{\widehat} 
\def \reg {\partial _{\kern1pt\text{reg}}} 

\def\dd{\operatorname{d}}
\def\dh{\widehat{\operatorname{d}}}
\def\clu#1#2{\operatorname{clust}_{#1^{**}}(#2)}

\def\de#1{\delta\left(#1\right)}

\def\nc#1{\operatorname{ca}\left(#1\right)}

\newcommand{\neimp}{\rotatebox{45}{$\Rightarrow$}}
\newcommand{\seimp}{\rotatebox{315}{$\hskip-10pt\Rightarrow$}} 

\begin{document}

\title[Quantitative Schur property]{On a difference between quantitative weak sequential completeness and the quantitative Schur property}
\author{O.F.K. Kalenda and J. Spurn\'y}

\address{Department of Mathematical Analysis \\
Faculty of Mathematics and Physic\\ Charles University\\
Sokolovsk\'{a} 83, 186 \ 75\\Praha 8, Czech Republic}
\email{kalenda@karlin.mff.cuni.cz}
\email{spurny@karlin.mff.cuni.cz}


\thanks{The authors were supported by the Research Project
MSM~0021620839 from the Czech Ministry of Education. The first author was moreover supported in part by the grant GAAV IAA 100190901. The second author was partly supported by the grant GA\v CR  201/07/0388.}

\begin{abstract}
We study quantitative versions of the Schur property and weak sequential completeness, proceeding thus with investigations started by G.~Godefroy, N.~Kalton and D.~Li and continued by H.~Pfitzner and the authors.
We show that the Schur property of $\ell_1$ holds quantitatively in the strongest possible way and construct an example of a Banach space which is quantitatively weakly sequentially complete, has the Schur property but fails the quantitative form of the Schur property.
\end{abstract}
\maketitle

%%%%%%%%%%%%%%%%%%%%%%%%%%%%%%%%%%%%%%%%%%%%%%%%%%%%
\section{Introduction}

We recall that a Banach space is \emph{weakly sequentially complete} if every weakly Cauchy sequence is weakly convergent. Further, a Banach space has the \emph{Schur property} if any weakly convergent sequence is norm convergent. This is easily seen to be equivalent to the fact that any weakly Cauchy sequence is norm Cauchy (and hence norm convergent), and thus any space with the Schur property is weakly sequentially complete. A classical example of a Banach space with the Schur property is the space $\ell_1$ of all absolutely summable sequences. An example of a weakly sequentially complete Banach space without the Schur property is the Lebesgue space $L_1(0,1)$.

In the present paper we study a quantitative version of the Schur property and its relation to quantitative weak sequential completeness. A quantitative version of weak sequential completeness was studied in \cite{GoKaLi2,gode-handbook,ka-pf-sp}. The existence of a quantitative Schur property was pointed out to us by M.~Fabian who observed in 2005 that the space $\ell_1$ has, in the terminology defined below, the $5$-Schur property. The referee pointed out to us that in the paper \cite{GoKaLi1} the authors use the notion of ``$1$-strong Schur property'' which is another quantification of the Schur property. We explain the relation to our notions in the last section. Before stating the known results and our contribution we shall define the quantitative properties.

For a bounded sequence $(x_k)$ in a Banach space $X$, we write
$\clu{X}{x_k}$ for the set of all weak* cluster points of $(x_k)$ in $X^{**}$ and by $\de{x_k}$ we denote its diameter. This quantity measures in a way how far the sequence is from being weakly Cauchy. Similarly, the quantity
\[
\nc{x_k}=\inf_{n\in\en} \diam\{x_k\colon k\geq n\}
\] 
measures how far the sequence is from being norm Cauchy.

Further, if  $A$, $B$ are  nonempty subsets of a Banach space $X$, then 
\[
\dd(A,B)=\inf\{\|a-b\|\colon a\in A, b\in B\}
\]
denotes the
usual distance between $A$ and $B$  and the Hausdorff non-symmetrized distance from $A$ to $B$ is defined by
\[
   \dh(A,B)= \sup\{\operatorname{d}(a,B): a\in A\}.
\]

We will say that a Banach space $X$ has the \emph{$C$-Schur property} (where $C\ge 0$) if
$$\nc{x_k}\le C \de{x_k}$$
for each bounded sequence $(x_k)$ in $X$. Similarly, $X$ is said to be \emph{$C$-weakly sequentially complete} if
$$\dh(\clu{X}{x_k},X)\le C \de{x_k}$$
for each bounded sequence $(x_k)$ in $X$. It is clear that any space with the $C$-Schur property has the Schur property and that any $C$-weakly sequentially complete space is weakly sequentially complete. It follows from Proposition~\ref{prop-easy} below that $C$-Schur property implies $C$-weak sequential completeness.

It is proved in \cite[Lemma IV.7]{GoKaLi2} that any $L$-embedded Banach space is $1$-weakly sequentially complete. We recall that a Banach space $X$ is called \emph{$L$-embedded} if there exists a projection $P:X^{**}\to X$ such that
\[
\|Px^{**}\|+\|(I-P)x^{**}\|=\|x^{**}\|,\quad x^{**}\in X^{**}.
\]
This result was  mentioned in \cite[p.\ 829]{gode-handbook} together with the question which weakly sequentially complete Banach spaces satisfy a quantitative version of this property. Recently some new results were obtained in \cite{ka-pf-sp}. It is proved there that any $L$-embedded space is $\frac12$-weakly sequentially complete, where the constant $\frac12$ is optimal as witnessed by the space $\ell_1$. The quoted paper further contains  an example of a~Schur space which is not $C$-weakly sequentially complete for any $C\ge0$. Conversely, $C$-weak sequential completeness does not imply Schur property (consider for example reflexive spaces or the space $L^1(0,1)$).

Inspired by the above mentioned results and a remark of M.\ Fabian, we investigate in this paper a quantification of the Schur property. We will assume that our Banach spaces are real. However, although some methods work only in real spaces, almost all the results are valid for complex spaces as well. We will discuss it in the final section.

The first proposition contains two easy inequalities and their consequence on the relationship of the quantitative Schur property and quantitative weak sequential completeness.

\begin{prop}\label{prop-easy} Let $(x_k)$ be a bounded sequence in a Banach space $X$. Then the following inequalities hold:
\begin{eqnarray}\label{ca-dh}
\dh(\clu{X}{x_k}, X)\leq \nc{x_k}, \\
\label{de-ca} \de{x_k}\leq \nc{x_k}.
\end{eqnarray}
In particular, if $X$ has the $C$-Schur property, then it is $C$-weakly sequentially complete.
\end{prop}

\begin{proof} We first observe that by the weak* lower semicontinuity of the norm we have
\[
\diam \{x_k\colon k\geq n\}=\diam \ov{\{x_k\colon k\geq n\}}^{w*},
\]
and thus
\[
\dd(x^{**},X)\leq \diam \ov{\{x_k\colon k\geq n\}}^{w*},\quad  x^{**}\in \clu{X}{x_k}, n\in\en.
\]
From this we deduce that
\[
\nc{x_k}=\inf_{n\in\en} \diam\{x_k\colon k\geq n\}=\inf_{n\in\en}\diam \ov{\{x_k\colon k\geq n\}}^{w*}\geq \dh(\clu{X}{x_k},X),
\]
and
\[
\nc{x_k}\geq \diam \bigcap_{n=1}^\infty \ov{\{x_k\colon k\geq n\}}^{w*}=\de{x_k}.
\]

This proves \eqref{ca-dh} and \eqref{de-ca}. The remaining statement follows immediately from \eqref{ca-dh}.
\end{proof}

Inequalities \eqref{de-ca} and \eqref{ca-dh} are optimal. Indeed, let $X=c_0$ and $(y_k)$ be the summing basis (i.e., $y_k=(1,\dots,1,0,\dots)$ where the last `$1$' is on the $k$-th place). Denote by $(x_k)$ the sequence
$$y_1,0,y_2,0,y_3,0,\dots$$
Then clearly $\de{x_k}=\nc{x_k}=\dh(\clu{X}{x_k},X)=1$.

Inequality \eqref{de-ca} shows in particular that a nontrivial space cannot have the $C$-Schur property for any $C<1$. Another trivial consequence is that in spaces with the $C$-Schur property the quantities $\delta$ and $\operatorname{ca}$ are equivalent.

Let us remark that the situation with weak sequential completeness is different -- any $L$-embedded space is $\frac12$-weakly sequentially complete by \cite{ka-pf-sp} and reflexive spaces are trivially $0$-weakly sequentially complete. The constant $\frac12$ is optimal by the following proposition.

\begin{prop}\label{reflexive} Let $X$ be a Banach space which is $C$-weakly sequentially complete
for some $C<\frac12$. Then $X$ is reflexive.

Moreover, even a stronger version holds: Suppose that there is $C<\frac12$ such that
$$\dd(\clu{X}{x_k},X)\le C \de{x_k}$$
for each bounded sequence $(x_k)$ in $X$. Then $X$ is reflexive.
\end{prop}

\begin{proof} It is clear that first part follows from the stronger version, so
let us show the stronger statement. Let $X$ be a non-reflexive Banach space.
Let $\varepsilon>0$ be arbitrary. 
It follows from \cite[Theorem 1]{GHP} that there is $x^*\in X^*$ such that any $x^{**}\in B_{X^{**}}$ such that $x^{**}(x^*)=\|x^*\|$ satisfies $\dd(x^{**},X)\ge 1-\varepsilon$. 

Indeed, suppose that for any $y^*\in X^*$ there is some $x^{**}=x^{**}_{y^*}\in B_{X^{**}}$ such that $x^{**}(y^*)=\|y^*\|$ and $\dd(x^{**},X)< 1-\varepsilon$. Let $B=\{x^{**}_{y^*}: y^*\in X^*\}$.
Then $B$ is a boundary for $X^*$ and $\dh(B,X)<1$, which contradicts  \cite[Theorem 1]{GHP}.

So, let $x^*$ be as in the first paragraph. 
Let $(x_k)$ be a sequence in $B_X$ such that $x^*(x_k)\to \|x^*\|$. Then clearly
$\de{x_k}\le 2$. Moreover, $\dd(\clu{X}{x_k},X)\ge 1-\varepsilon$, as any 
$x^{**}\in\clu{X}{x_k}$ satisfies $x^{**}(x^*)=\|x^*\|$.
This completes the proof.
\end{proof}

Let us remark, that the converse of the previous proposition is trivially valid as well, as any reflexive space is $0$-weakly sequentially complete.

We continue by our first main result which is an improvement of Schur's theorem for $\ell_1$.  

\begin{thm}\label{positive} The space $\ell_1$ has the $1$-Schur property,
i.e., 
\[
\nc{x_k}= \de{x_k}.
\]
for any bounded sequence $(x_k)$ in $\ell_1$. 
\end{thm}

It follows from \cite{ka-pf-sp} that there is a Banach space with the Schur property which fails its quantitative version. Indeed, it is constructed there a space with the Schur property which is not $C$-weakly sequentially complete for any $C\ge 0$. It follows from Proposition~\ref{prop-easy} that the same space fails the $C$-Schur property for each $C\ge0$.
Our second result shows that a quantitative version of the Schur property  may fail even for an $L$-embedded space (which is $\frac12$-weakly sequentially complete).

\begin{example}\label{el-ex}
There exists a separable $L$-embedded Banach space with the Schur property 
which fails the $C$-Schur property for every $C\ge0$.
\end{example}

Following the suggestion of the referee, we include the following diagramm of implications between the classes of Banach spaces we study in this paper (``wsc'' is an abbreviation of ``weakly sequentially complete'').

$$\begin{array}{ccccccc}
&&&&\mbox{Schur}&&\\
&&&\neimp&&\seimp&\\
C\mbox{-Schur}& \Rightarrow & C\mbox{-wsc and Schur} & \Rightarrow & C\mbox{-wsc} & \Rightarrow & \mbox{wsc} 
\end{array}$$

All these implications are easy to check. If a constant $C$ is included both in the assumption and in the conclusion, the constant in the conclusion has the same value.

None of the displayed implications can be reversed, even if we allow the respective constants to change.
For the first implication from the left, it follows from Example~\ref{el-ex}. A reasoning for the implication $\neimp$ and the third one in the bottom line is a result of \cite{ka-pf-sp}. For the remaining two implications one can use the classical space $L^1(0,1)$.

%%%%%%%%%%%%%%%%%%%%%%%%%%%%%%%%%%%%%%%%%%%%%%%%%%%%
\section{Proof of Theorem~\ref{positive}}

We start the proof with a simple fact based upon the standard ``sliding hump" argument. If $x\in\ell_1$ and $N\subset \en$, we denote by $x\r_N$ the element arising from $x$ by setting its coordinates outside $N$ to zero, i.e.,
$$x\r_N(n)=\begin{cases} x(n), & n\in N, \\ 0, & n\in\en\setminus N.\end{cases}$$

\begin{lemma}\label{hump}
Let $(y_n)$ be a  sequence in $\ell_1$  converging pointwise to $y$ with $\|y\|<\ep$. Then there exists a subsequence $(y_{n_k})$ and indices $0=N_0<N_1<N_2<\cdots$ such that
\[
\|y_{n_k}\r_{(N_{k-1}, N_k]}\|>\|y_{n_{k}}\|-\ep,\quad k\in\en.
\]
\end{lemma}

\begin{proof}
First we set $N_0=0$, $n_1=1$ and choose $N_1>0$ such that $\|y_{n_1}\r_{(0, N_1]}\|>\|y_{n_{1}}\|-\ep$.
 Suppose now that $k\in\en$ is such that we have already constructed $n_j$ for $1\le j\le k$ and $N_j$ for $0\le j\le k$.

Since $(y_n)$ converges pointwise to $y$, we can select $n_{k+1}>n_k$  so large that $\|y_n\r_{(N_0, N_k]}\|<\ep$ for every $n\geq n_{k+1}$. Thus, in particular 
\[
\|y_{n_{k+1}}\r_{(N_k,+\infty)}\|>\|y_{n_{k+1}}\|-\ep,
\]
so we can find $N_{k+1}>N_{k}$ such that 
$$\|y_{n_{k+1}}\r_{(N_{k},N_{k+1}]}\|>\|y_{n_{k+1}}\|-\ep.$$ 
This completes the construction.\end{proof}

Now we proceed with the proof of Theorem~\ref{positive}. Let $(x_k)$ be a bounded sequence in $X=\ell_1$ and $\ep>0$. We consider an arbitrary $c<\nc{x_k}$. 

We extract subsequences $(a_n)$ and $(b_n)$ from $(x_k)$ such that $c<\|a_n-b_n\|$ for $n\in\en$. Denoting $y_n=a_n-b_n$, we pass to a subsequence if necessary and assume thus  that $(y_n)$ pointwise converges to $y\in\ell_1$. Let 
%$c_1=\|y\|$ and 
$m\in\en$ be chosen  such that $\|y\r_{(m,+\infty)}\|<\ep$. By omitting finitely many elements of $(y_n)$ we achieve that 
\begin{equation}\label{eNko}
\|(y_n-y)\r_{[1,m]}\|<\ep,\quad n\in\en.
\end{equation}
Hence
\begin{equation}
\label{cN}
c<\|y_n\r_{[1,m]}\|+\|y_n\r_{(m,+\infty)}\|\leq \|y\r_{[1,m]}\|+\ep+\|y_n\r_{(m,+\infty)}\|,\quad n\in\en.
\end{equation}

Using Lemma~\ref{hump} applied to $(y_n\r_{(m,+\infty)})$ we obtain a subsequence $(y_{n_k})$ and indices $m=N_0<N_1<\cdots$ such that 
\begin{equation}
\label{hrb}
\|y_{n_k}\r_{(N_{k-1}, N_k]}\|>\|y_{n_{k}}\r_{(m,+\infty)}\|-\ep, 
\quad k\in\en. 
\end{equation}
Let $x^*\in\ell_\infty=\ell_1^*$ be defined as
\[
x^{*}(j)=\begin{cases} \operatorname{sign} y(j),& j\in[1,m],\\
                     \operatorname{sign} y_{n_k}(j),& j\in(N_{k-1}, N_k],\  k\in\en.
       \end{cases}              
\]
Then $\|x^*\|\le 1$, and, for each $k\in\en$, we have 
using~\eqref{eNko}, \eqref{hrb} and \eqref{cN}
\[
\aligned
x^*(y_{n_k})&=x^*(y_{n_k}\r_{[1,m]})+x^*(y_{n_k}\r_{(N_{k-1}, N_k]})+
\sum_{j\in(m,+\infty)\setminus(N_{k-1},N_k]} x^*(j)y_{n_k}(j)\\
            &> x^*(y\r_{[1,m]})-\ep+\|y_{n_{k}}\r_{(N_{k-1},N_k]}\|-\|y_{n_k}\r_{(m,+\infty)\setminus(N_{k-1},N_k]}\|\\
            &> \|y\r_{[1,m]}\|+\|y_{n_k}\r_{(m,+\infty)}\|-3\ep\\
            &> c-4\ep.
        \endaligned
\]

Therefore we have 
\[
x^*(a_{n_k}-b_{n_k})\geq c-4\ep.
\]
Up to a passing to a subsequence we can suppose that the sequence $(x^*(a_{n_k}))$ converges. Up to passing to a further  subsequence we can assume that the sequence $(x^*(b_{n_k}))$ converges as well.
Let $a^{**}$ and $b^{**}$ be  weak* cluster points of $(a_{n_k})$ and $(b_{n_k})$, respectively. Then $a^{**}, b^{**}\in\clu{X}{x_k}$ and
\[
\|a^{**}-b^{**}\|\geq (a^{**}-b^{**})(x^*)=\lim_{k\to\infty} x^*(a_{n_k}-b_{n_k}) \geq c-4\ep.
\] 
Since $c$ and $\ep$ are arbitrary, $\nc{x_k}\leq \de{x_k}$. 

Inequality \eqref{de-ca} from Proposition~\ref{prop-easy} finishes the proof.

%%%%%%%%%%%%%%%%%%%%%%%%%%%%%%%%%%%%%%%%%%%%%%%%%%%%
\section{Construction of Example~\ref{el-ex}}

The construction is based upon Example~4 of \cite{ka-pf-sp}. It uses a standard renorming technique
(see, e.g. \cite[Proposition III.2.11]{hawewe}). 
We recall that $\beta\en$ is the \v{C}ech--Stone compactification of $\en$ and $M(\beta\en)$ is the space of all signed Radon measures on $\beta\en$ considered as the dual of $\ell_\infty$.

Let us fix $\alpha>0$ and consider the space 
\[
Y_\alpha=(\ell_1,\alpha\|\cdot\|_1)\oplus_1 (\ell_2,\|\cdot\|_2).
\]
Here $\|\cdot\|_1$ and $\|\cdot\|_2$ denote the usual norms on $\ell_1$ and $\ell_2$, respectively. 
Note that we have the following canonical identifications:
$$\begin{aligned}
Y_\alpha^*&=(\ell_\infty,\tfrac1\alpha\|\cdot\|_\infty)\oplus_\infty (\ell_2,\|\cdot\|_2), \text{ and}\\
Y_\alpha^{**}&=(M(\beta\en),\alpha\|\cdot\|_{M(\beta\en)})\oplus_1 (\ell_2,\|\cdot\|_2).
\end{aligned}$$

For $k\in\en$, let $x_k=(e_k,e_k)\in Y_\alpha$, where $e_k$ denotes the $k$-th canonical basic vector.
Let $X_\alpha$ be the closed linear span of the set $\{x_k:k\in\en\}$.

We claim that
\begin{equation}\label{Xalpha}
X_\alpha=\left\{(x,y)\in Y_\alpha:
x(n)=y(n), n\in\en\right\}.
\end{equation}
Indeed, the set on the right-hand side is a closed linear subspace of $Y_\alpha$ containing $x_k$ for each $k\in\en$ which verifies the inclusion `$\subset$'. To prove the converse one, let us take any element $(z,y)\in Y_\alpha$ satisfying $z(n)=y(n)$ for all $n\in\en$. Since $z\in\ell_1$, we get
\[
(z,y)=\sum_{k=1}^\infty  z(k)x_k\in X_\alpha
\]
as the series is absolutely convergent.

Let $T:\ell_1\to\ell_2$ denote the identity mapping. Since $T$ maps the unit ball of $\ell_1$ into the unit ball of $\ell_2$, we get $\|T\|\le 1$.
Therefore, for an arbitrary element $(x,y)\in X_\alpha$, we have
$$\alpha\|x\|_1 \le \|(x,y)\|_{X_\alpha}=\alpha\|x\|_1+\|Tx\|_2\le (\alpha+1)\|x\|_1.$$
%If $c_1,\dots, c_n$ are arbitrary numbers, then 
%\[
%\alpha \sum_{k=1}^n |c_k|\leq \|\sum_{k=1}^n c_kx_k\|_{X_\alpha}=\alpha \sum_{k=1}^n |c_k|+\left(\sum_{k=1}^n |c_k|^2\right)^{\frac12}\leq (\alpha+1) \sum_{k=1}^n |c_k|.
%\]
Thus the projection on the first coordinate is an isomorphism of $X_\alpha$ onto $\ell_1$. In particular, $X_\alpha$ has the Schur property.

We further observe that $X_\alpha^{**}$ is canonically identified with the weak* closure of $X_\alpha$ in $Y_\alpha^{**}$, thus
\begin{equation}
\label{Xalpha**} X_\alpha^{**}=\{(\mu,y)\in M(\beta\en)\times \ell_2: 
 \mu(n)=y(n), n\in\en\}.
\end{equation}
Indeed, the set on the right-hand side is  a weak* closed linear subspace of $Y_\alpha^{**}$ containing $X_\alpha$, which proves the inclusion `$\subset$'.
To prove the converse one, we fix $(\mu,y)$ in the set on the right-hand side.
Take a bounded net $(u_\tau)$ in $\ell_1$ which weak* converges to $\mu$. For each $\tau$ there is a unique $y_\tau\in \ell_2$ such that $(u_\tau,y_\tau)\in X_\alpha$. Then $(y_\tau)$ is clearly a bounded net in $\ell_2$. Moreover, we will show that $(y_\tau)$ weak* (i.e. weakly) converges to $y$. Since the weak topology  on bounded sets in $\ell_2$ coincides with the topology of pointwise convergence, it suffices to show that $y_\tau$ pointwise converge to $y$. Indeed,
\[
y_\tau(n)=\mu_\tau(n) \to \mu(n)=y(n),\quad n\in\en.
\]

It follows that $X_\alpha$ is $L$-embedded in $X_\alpha^{**}$ because the projection $P:X_\alpha^{**}\to X_\alpha$ defined as
\[
P(\mu,y)=(\mu\r_\en,y),\quad (\mu,y)\in M(\beta\en)\times \ell_2,
\]
satisfies 
\[
\|(I-P)(\mu,y)\|_{X_\alpha^{**}}+\|P(\mu,y)\|_{X_\alpha^{**}}=\|(\mu,y)\|_{X_\alpha^{**}},\quad (\mu,y)\in X_\alpha^{**}.
\]

Further, for the sequence $(x_k)$, its weak$^*$ cluster points in 
$X_\alpha^{**}$
are equal to
\[
\{(\ep_t, 0): t\in \beta\en\setminus\en\},
\]
where $\ep_t$ denotes the Dirac measure at a point $t\in\beta\en$.

We claim that, for our sequence $(x_k)$, we have 
\begin{equation}
\label{est-alpha}
\nc{x_k}=2\alpha+\sqrt{2}\quad\text{and}\quad \de{x_k}=2\alpha.
\end{equation}
To see the first inequality, we observe that for each distinct $k,k'\in\en$ we have
\[
\|x_k-x_{k'}\|_{X_\alpha}=\alpha\|e_k-e_{k'}\|_1+\|e_k-e_{k'}\|_{2}=2\alpha+\sqrt{2}.
\]
The second one follows from the fact that, given $t,t'\in\beta\en\setminus\en$ distinct, then 
\[
\|(\ep_t,0)-(\ep_{t'},0)\|_{X_\alpha^{**}}=\|(\ep_t-\ep_{t'},0)\|_{X_\alpha^{**}}
=\alpha\|\ep_t-\ep_{t'}\|_{M(\beta\en)}=2\alpha.
\]
This verifies \eqref{est-alpha}.

Now we use the described procedure to construct the desired space $X$.
For $n\in\en$, let $\alpha_n=\frac1n$ and let $X_{\frac1n}$ be the space constructed for $\alpha_n$. Let 
\[
X=\left(\sum_{n=1}^\infty X_{\frac1n}\right)_{\ell_1}
\]
be the $\ell_1$-sum of the spaces $X_{\frac1n}$. We claim that $X$ is the  required space. 

First, since each $X_{\frac1n}$ has the Schur property, $X$, as their $\ell_1$-sum, possesses this property as well (this follows
by a straightforward modification of the proof that $\ell_1$ has the Schur property, see \cite[Theorem~5.19]{fhhmpz}).

Similarly, as an $\ell_1$-sum of $L$-embedded spaces, $X$ is $L$-embedded as well (see \cite[Proposition~1.5]{hawewe}).

Finally, fix $n\in\en$.  We consider a sequence $\wh{x}_k=(0,\dots,0,\stackrel{n\text{-th}}{x_k},0,\dots)$,
where the elements $x_k\in X_{\frac1n}$, $k\in\en$, are defined above.
Then, for any $k,k'\in\en$ distinct, 
\[
\|\wh{x}_k-\wh{x}_{k'}\|_X= \|x_k-x_{k'}\|_{X_{\frac1n}},
\] 
hence $\nc{\wh{x}_k}=\nc{x_k}>\sqrt2$ by \eqref{est-alpha}.

On the other hand,
\[
\de{\wh{x}_k}= \de{x_k}=\frac2n,
\]
again by \eqref{est-alpha}. So, 
\[
\nc{\wh{x}_k}>\frac n{\sqrt2}\;\de{\wh{x}_k}.
\]
Since $n\in\en$ is arbitrary, the conclusion follows.

\section{Final remarks}

In \cite[p. 57]{GoKaLi1} a Banach space is said to have the {\it $1$-strong Schur property}, whenever for any $\delta\in(0,2]$, any $\varepsilon>0$ and any normalized $\delta$-separated sequence in $X$ there is a subsequence which is $(\frac2\delta+\varepsilon)$-equivalent to the standard $\ell_1$ basis.

Although this property has adjective ``strong'', it is in fact weaker than our $1$-Schur property, as witnessed by the following proposition. We do not know whether these two properties are in fact equivalent but we conjecture that it is not the case. The reason for this opinion is the fact that the $1$-Schur property is a property of all bounded sequences while the $1$-strong Schur property is a property of very special sequences.

\begin{prop}\label{posledni} 
\begin{itemize}
	\item Any Banach space with the $1$-Schur property has the $1$-strong Schur property.
	\item Any Banach space with the $1$-strong Schur property has the $5$-Schur property.
\end{itemize}
\end{prop}

\begin{proof} Suppose that $X$ is a Banach space with the $1$-Schur property and fix 
$\delta\in(0,2]$, $\varepsilon>0$ and a normalized $\delta$-separated sequence $(x_k)$ in $X$.
Since $(x_k)$ is $\delta$-separated, $\nc{x_{k_n}}\ge\delta$ for each subsequence $(x_{k_n})$ of $(x_k)$.
By the $1$-Schur property we get $\de{x_{k_n}}\ge\delta$ for each subsequence. By \cite[Theorem 3.2]{behrends} there is a subsequence $(x_{k_n})$ such that
$$\left\|\sum_{n=1}^N \alpha_n x_{k_n}\right\|\ge \frac1{\frac2\delta+\varepsilon} \sum_{n=1}^N |\alpha_n|$$
for each $N\in\en$ and each choice of real numbers $\alpha_1,\dots,\alpha_N$. Note, that we used that
$$\frac1{\frac2\delta+\varepsilon}<\frac\delta2.$$
It follows that $(x_{k_n})$ is $(\frac2\delta+\varepsilon)$-equivalent to the standard $\ell_1$ basis.

Conversely, suppose that $X$ is a Banach space which has the $1$-strong Schur property. Let $(x_k)$ be a bounded sequence in $X$ and $c>0$ be such that $\nc{x_k}>5c$. We will show that $\de{x_k}\ge c$.
We will distinguish two cases:
\begin{itemize}
	\item[(a)] There is $n\in\en$ such that for each $k\in\en$ we have $\dd(x_k,\{x_1,\dots,x_n\})\le c$.
	\item[(b)] For each $n\in\en$ there is $k>n$ such that $\dd(x_k,\{x_1,\dots,x_n\})> c$.
\end{itemize}
It is clear that exactly one of these two cases takes place.

First suppose that the case (a) occurs. As $\nc{x_k}>5c$, there are two subsequences $(u_k)$ and $(v_k)$ of $(x_k)$ such that $\|u_k-v_k\|>5c$ for each $k\in\en$. For each $k\in \en$ fix $p_k,q_k\in \{1,\dots,n\}$ such that $\|u_k-x_{p_k}\|\le c$ and $\|v_k-x_{q_k}\|\le c$. Up to taking a subsequence we can assume that the sequence $(p_k)$ is constant. Denote the common value by $p$. Up to taking a further subsequence we can suppose that the sequence $(q_k)$ is constant as well. Let $q$ be the common value. Set $a=x_p$ and $b=x_q$. Let $u^{**}$ be a weak* cluster point of $(u_k)$  in $X^{**}$.
Similarly, let $v^{**}$ be a weak* cluster point of $(v_k)$. Then
$$\begin{aligned}
\|v^{**}-u^{**}\|&\ge\|b-a\|-\|v^{**}-b\|-\|u^{**}-a\|
\ge\|b-a\|-2c\\&\ge\|v_1-u_1\|-\|b-v_1\|-\|a-u_1\|-2c>5c-4c=c.
\end{aligned}$$
Hence $\de{x_k}>c$.

Next suppose that the case (b) occurs and fix an arbitrary $\varepsilon>0$. It is easy to construct a $c$-separated subsequence $(y_k)$ of $(x_k)$. As $(y_k)$ is bounded, without loss of generality we can suppose that the sequence $(\|y_k\|)$ converges to some $\alpha\ge 0$. As the sequence is $c$-separated, necessarily $\alpha\ge\frac c2$. Up to omitting a finite number of $(y_k)$ we can suppose that
$|\|y_k\|-\alpha|<\varepsilon\alpha$ for each $k\in\en$. Then the sequence $(\frac{y_k}{\|y_k\|})$ is a normalized sequence which is $\frac c{\alpha}-2\varepsilon$-separated. Indeed, for any distinct $k,l\in\en$ we have
$$\begin{aligned}\left\|\frac{y_k}{\|y_k\|}-\frac{y_l}{\|y_l\|}\right\| &\ge 
\left\|\frac{y_k}{\alpha}-\frac{y_l}{\alpha}\right\| -
\left\|\frac{y_k}{\|y_k\|}-\frac{y_k}{\alpha}\right\|-
\left\|\frac{y_l}{\|y_l\|}-\frac{y_l}{\alpha}\right\| \\ &
>\frac c\alpha - \left|\frac{\alpha-\|y_k\|}{\alpha}\right| -
\left|\frac{\alpha-\|y_l\|}{\alpha}\right|>\frac c\alpha-2\varepsilon.\end{aligned}$$
By the $1$-strong Schur property we can extract a subsequence of $\left(\frac{y_k}{\|y_k\|}\right)$
which is $\left(\frac{2}{\frac c\alpha-2\varepsilon}+\varepsilon\right)$-equivalent to the standard $\ell_1$-basis.
Thus
$$\de{\frac{y_k}{\|y_k\|}}\ge \frac{2}{\frac{2}{\frac c\alpha-2\varepsilon}+\varepsilon}$$
by \cite[Lemma 5]{ka-pf-sp}.
Since $\varepsilon>0$ is arbitrary, we get $\de{\frac{y_k}{\|y_k\|}}\ge \frac c\alpha$. It follows
that $\de{y_k}\ge c$, thus $\de{x_k}\ge c$. This completes the proof. 
\end{proof}
%\bibliography{quantschur }\bibliographystyle{plain}

As we remarked in the beginning of the paper, we worked only with real spaces. However, all the results except possibly for Proposition~\ref{posledni} are true also for complex spaces. Let us explain it.

Let $X$ be a complex Banach space. Denote by $X_R$ the real version of $X$, i.e., the same space considered over the reals. Given a bounded sequence $(x_k)$ in $X$, all the considered quantities ($\dd$, $\dh$, $\delta$, $\operatorname{ca}$) are the same for $X$ and for $X_R$. For the quantities $\dd$, $\dh$ and $\delta$ it is explained in \cite[Section 5]{ka-pf-sp}, for the quantity $\operatorname{ca}$ it is trivial. Hence, in particular, $X$ is $C$-weakly sequentially complete (or has the $C$-Schur property) if and only if $X_R$ has this property. Thus Propositions~\ref{prop-easy} and~\ref{reflexive} are valid for complex spaces as well. 

The proof of Theorem~\ref{positive} can be easily adapted to the complex case. It is enough to define $x^*$ using complex signs (i.e., appropriate complex units) and to estimate from below the real part of $x^*(y_{n_k})$. The proof of Example~\ref{el-ex} works in the complex case without any change.

As for Proposition~\ref{posledni}, the second part works for complex spaces without any change. It is not clear whether the first part is valid because the proof uses a result of \cite{behrends} which works for real spaces and it is an open problem whether it is valid for complex spaces as well. It is also not clear, whether the $1$-strong Schur property of $X$ is equivalent to that of $X_R$.

\section*{Acknowledgement}

We are grateful to the referee for helpful comments which led to an improvement of our paper.

\end{document}